\documentclass[12pt]{amsart}
\usepackage[margin=1.3in]{geometry}
\usepackage[OT2,T1]{fontenc}
\usepackage[utf8]{inputenc}
\usepackage{tikz-cd,stmaryrd}
\usetikzlibrary{graphs,decorations.pathmorphing,decorations.markings}
\usepackage{microtype,hyperref,amsfonts,amssymb,amsthm,mathrsfs}
\usepackage[alphabetic]{amsrefs}

\theoremstyle{plain}
\newtheorem{theorem}{Theorem}[section]
\newtheorem{proposition}[theorem]{Proposition}
\newtheorem{lemma}[theorem]{Lemma}

\newtheorem{corollary}[theorem]{Corollary}
\newtheorem*{conjecture}{Conjecture}
\newtheorem*{hypothesis}{Hypothesis}
\theoremstyle{definition}
\newtheorem{definition}{Definition}[section]
\newtheorem{instance}[definition]{Example}

\newtheorem{remark}[definition]{Remark}

\usepackage{enumitem}
\setlist[itemize]{leftmargin=2em}
\numberwithin{equation}{section}
\title[Mirror duality between
CY fractional complete intersections]{Mirror duality between Calabi--Yau fractional complete intersections}

\author[Tsung-Ju~Lee]{Tsung-Ju~Lee}
\address{Tsung-Ju~Lee: Center of Mathematical Sciences and Applications, 20 Garden St., Cambridge, MA 02138, U.S.A.}
\email{tjlee@cmsa.fas.harvard.edu}

\begin{document}
\begin{abstract}
This is an expanded version of the author's talk 
at the third annual meeting of International Consortium of Chinese Mathematicians
held at USTC in December 2020\footnotemark.
In this expository article, we give a survey on 
joint works with Hosono, Lian, and Yau
\cites{2020-Hosono-Lee-Lian-Yau-mirror-symmetry-for-double-
cover-calabi-yau-varieties,2022-Lee-Lian-Yau-on-calabi-yau-fractional-complete-intersections}.
We also carry out explicit examples to illustrate
the results in enumerative geometry which will appear in our forthcoming papers. 

\end{abstract}
\maketitle
\tableofcontents
\stepcounter{footnote}
\footnotetext{The ICCM 
\href{https://math.ustc.edu.cn/2020/1225/c18653a466436/pagem.htm}{website}.}

\section{Introduction}
Mirror symmetry for singular Calabi--Yau varieties 
was discovered by Hosono, Lian, Takagi, and Yau
in their recent work on the study of the family of \(K3\) surfaces arising
from double covers over \(\mathbf{P}^{2}\)
branched along six lines in general position
\cites{2020-Hosono-Lian-Takagi-Yau-k3-surfaces-from-configurations-of-six-lines-in-p2-and-mirror-symmetry-i,2019-Hosono-Lian-Yau-k3-surfaces-from-configurations-of-six-lines-in-p2-and-mirror-symmetry-ii-lambda-k3-functions},
which were investigated by Matsumoto, Sasaki, and Yoshida 
as a higher dimensional analogue of the Legendre family
\cites{1988-Matsumoto-Sasaki-Yoshida-the-period-map-of-a-4-parameter-family-of-k3-surfaces-and-the-aomoto-gelfand-hypergeometric-function-of-type-3-6,1992-Matsumoto-Sasaki-Yoshida-the-monodromy-of-the-period-map-of-a-4-parameter-family-of-k3-surfaces-and-the-hypergeometric-function-of-type-3-6}.
Such a \(K3\) surface is singular and admits 
\(15\) ordinary double points (ODPs);
blowing up at these points gives a crepant resolution.
Denote by \([x\mathpunct{:}y\mathpunct{:}z]\) 
the homogeneous coordinates on \(\mathbf{P}^{2}\).
The branch locus can be parameterized by linear functions
\(a_{1j}x+a_{2j}y+a_{3j}z\) with \(j=1,\ldots,6\).
Regarding \((a_{ij})\) as a matrix, 
we see that the configurations
of six lines are parameterized by a GIT quotient
\begin{equation*}
P(3,6):=\mathrm{GL}_{3}(\mathbb{C}) \backslash \mathrm{M}(3,6) \slash
(\mathbb{C}^{\ast})^{6}.
\end{equation*}
Here \(\mathrm{M}(3,6)\) is an open subset 
in \(\mathrm{Mat}_{3\times 6}(\mathbb{C})\)
consisting of matrices whose any \(3\times 3\) minors are invertible.
The group \(\mathrm{GL}_{3}(\mathbb{C})\)
acts on \(\mathrm{M}(3,6)\) via the usual multiplication 
on the left and
\((\mathbb{C}^{\ast})^{6}\) acts on \(\mathrm{M}(3,6)\)
via scaling the columns of elements in \(\mathrm{M}(3,6)\).

The parameter space \(P(3,6)\) admits two compactifications --
\textbf{(a)} a GIT compactification 
(a.k.a.~the Baily--Borel--Satake compactification)
\cites{1988-Dolgachev-Ortland-point-sets-in-projective-spaces-and-theta-functions,1993-Matsumoto-theta-functions-on-bounded-symmetric-domain-of-type-i-22-and-the-period-map-of-a-4-parameter-family-of-k3-surfaces}
and \textbf{(b)} a toridal compactification constructed by Reuvers
\cite{2006-Reuvers-moduli-spaces-of-configurations}.
However, as Hosono, Lian, Takagi, and Yau pointed out, it is not clear whether or not
these compactifications admit a priori
the so-called \emph{large complex structure limit points} (LCSL points). 
In order to study mirror symmetry, they constructed a 
new compactification of \(P(3,6)\) instead and found 
LCSL points on it. 
We briefly explain their idea.
The \(\mathrm{GL}_{3}(\mathbb{C})\)-action 
on \(\mathbf{P}^{2}\) allows us to
rearrange three out of the six lines to the coordinate axes so that
the \(K3\) family is in fact parameterized by three lines 
in \(\mathbf{P}^{2}\). This procedure is 
called the \emph{partial gauge fixing}
in \cite{2020-Hosono-Lian-Takagi-Yau-k3-surfaces-from-configurations-of-six-lines-in-p2-and-mirror-symmetry-i}. 
The \(\mathrm{GL}_{3}(\mathbb{C})\times(\mathbb{C}^{\ast})^{6}\) action
is reduced to a \((\mathbb{C}^{\ast})^{5}\) action.
It then follows that the period integrals of
the \(K3\) family satisfy certain GKZ \(A\)-hypergeometric system
with an integral matrix \(A\in\mathrm{Mat}_{5\times 9}(\mathbb{Z})\) and
a \emph{fractional} exponent \(\beta\in\mathbb{Q}^{5}\).
The matrix \(A\) can be recognized as the integral matrix
associated to certain nef-partition on the base \(\mathbf{P}^{2}\)
and the torus \((\mathbb{C}^{\ast})^{5}\) can be identified with
\(L\otimes\mathbb{C}^{\ast}\), where \(L\) is the lattice relation of \(A\).
Consequently, \(P(3,6)\) admits a toroidal compactification
via the secondary fan.
It turns out that the standard techniques for 
Calabi--Yau hypersurfaces and complete intersections
in toric varieties are still applicable and 
results in \cites{1995-Hosono-Klemm-Theisen-Yau-mirror-symmetry-mirror-map-and-applications-to-calabi-yau-hypersurfaces,1996-Hosono-Lian-Yau-gkz-generalized-hypergeometric-systems-in-mirror-symmetry-of-calabi-yau-hypersurfaces,1997-Hosono-Lian-Yau-maximal-degeneracy-points-of-gkz-systems} 
can be adapted into the present situation.
Because of this striking similarity with the classical complete intersections, 
we shall call such a double cover a \emph{fractional complete intersection}.
Based on numerical evidences, 
it is conjectured that the mirror of the \(K3\) family is 
given by certain family of 
double covers over a del Pezzo surface of degree \(6\),
which is a blow-up of \(\mathbf{P}^{2}\)
at three torus invariant points 
(\cite{2019-Hosono-Lian-Yau-k3-surfaces-from-configurations-of-six-lines-in-p2-and-mirror-symmetry-ii-lambda-k3-functions}*{Conjecture 6.3}).
Note that such a del Pezzo surface can be obtained from
Batyrev--Borisov's duality construction for the nef-partition
on \(\mathbf{P}^{2}\) associated with the integral matrix \(A\)
and the conjectured mirror is constructed by taking an
appropriate double cover over it.

The purpose of this paper is to 
introduce the results 
obtained by Hosono, Lian, Yau and the author in
\cites{2020-Hosono-Lee-Lian-Yau-mirror-symmetry-for-double-cover-calabi-yau-varieties,2022-Lee-Lian-Yau-on-calabi-yau-fractional-complete-intersections}.
To summarize,
we generalized the construction 
to higher dimensional bases
to produce a pair of \emph{singular}
Calabi--Yau varieties
which is conjectured to be a mirror pair.
To support the conjecture, the 
first step is the topological test: one computes 
their Euler characteristics as well as their Hodge numbers;
these have been done in
\cite{2020-Hosono-Lee-Lian-Yau-mirror-symmetry-for-double-cover-calabi-yau-varieties}.
The second step is the quantum test: in the \(3\)-fold case, one could
carry out the \(A\) and \(B\) model correlation functions
and show that they are related under the \emph{mirror map}.
Here the \(B\) model is taken to be the 
variation of Hodge structures for
the equisingular family 
whereas the \(A\) model is taken to be the 
untwisted part of the genus zero orbifold Gromov--Witten
theory since our Calabi--Yau double covers are orbifolds.
In the present case, the period integrals
are governed by a GKZ \(A\)-hypergeometric system
with a fractional exponent. 
Mimicking the classical case, we found a close
relationship between the principal parts
of the operators in the GKZ \(A\)-hypergeometric system
and the cohomology ring of the base of the 
conjectured mirror Calabi--Yau variety;
this leads to a cohomology-valued series
first introduced in 1994 by Hosono, Lian, and Yau
(a.k.a.~Givental's \(I\)-function, up to 
an overall \(\Gamma\)-factor) 
which plays a crucial role in mirror symmetry
\cite{2022-Lee-Lian-Yau-on-calabi-yau-fractional-complete-intersections}.

In this expository paper, we will 
explain our idea and illustrate our results 
by carrying out explicit examples.

\subsection*{Acknowledgment}
The author would like to thank the committee of
International Consortium of Chinese Mathematicians
for giving him an opportunity to speak at the annual meeting.
He would like to thank Professors Shinobu Hosono, 
Bong H.~Lian, and Shing-Tung Yau for valuable communications
and suggestions. He also
thanks Center of Mathematical Sciences and Applications
at Harvard for hospitality while working on this project. 

\section{A singular double cover construction}

\subsection{Batyrev--Borisov's duality construction}
\label{subsection:B-B-duality-construction}
We review the construction of classical 
mirror pairs of Calabi--Yau complete
intersections in toric varieties given by
Batyrev and Borisov
\cite{1996-Batyrev-Borisov-on-calabi-yau-complete-intersections-in-toric-varieties}.
The dual pairs will be served as 
bases of singular Calabi--Yau double covers.
Let us fix the following notation and terminologies.
\begin{itemize}
  \item Let \(N=\mathbb{Z}^n\) be a lattice of rank \(n\) 
  and \(M:=\mathrm{Hom}_{\mathbb{Z}}(N,\mathbb{Z})\) be the dual lattice. 
  We denote by \(N_\mathbb{R}\) and \(M_\mathbb{R}\) the tensor products
  \(N\otimes_{\mathbb{Z}}\mathbb{R}\) and \(M\otimes_{\mathbb{Z}}\mathbb{R}\).
  \item For a complete fan \(\Sigma\) in \(N_{\mathbb{R}}\), 
  we denote by \(\Sigma(k)\) 
  the set of all \(k\)-dimensional cones in \(\Sigma\). 
  For \(\rho\in \Sigma(1)\),
  the same notation \(\rho\) is also used to denote 
  the primitive generator of the corresponding 1-cone. 
  \item The toric variety defined by \(\Sigma\) is 
  denoted by \(X_{\Sigma,N}\), \(X_{\Sigma}\) 
  or simply by \(X\) if the context is clear. 
  Let \(T=(\mathbb{C}^{\ast})^n\) be its maximal torus.
  Denote by \(D_{\rho}\) the Weil divisor on \(X_{\Sigma}\)
  determined by \(\rho\in\Sigma(1)\).
  \item Let \(D=\sum_{\rho} a_{\rho} D_{\rho}\) be a torus invariant divisor. 
  The \emph{polytope of \(D\)} is defined to be the set
  \begin{equation*}
    \Delta_D:=\{m\in M_{\mathbb{R}}~|~\langle m,\rho\rangle\ge -a_{\rho},
    ~\forall \rho\in\Sigma(1)\}.
  \end{equation*}
  The integral points in \(\Delta_{D}\) give a canonical
  basis of \(\mathrm{H}^{0}(X,D)\).
  \item A polytope in \(M_{\mathbb{R}}\) is called lattice polytope
  if its vertices belong to \(M\). For a lattice polytope \(\Delta\)
  in \(M_{\mathbb{R}}\), we denote by \(\Sigma_{\Delta}\) the normal fan of 
  \(\Delta\). The toric variety determined by 
  \(\Delta\) is denoted by \(\mathbf{P}_{\Delta}\). 
  We have \(\mathbf{P}_{\Delta}=X_{\Sigma_{\Delta}}\).
  \item A \emph{reflexive polytope} \(\Delta\subset M_{\mathbb{R}}\) 
  is a lattice polytope 
  containing the origin \(0\in M_{\mathbb{R}}\) in 
  its interior and such that the polar dual 
  \(\Delta^{\vee}\) is again a lattice polytope. If
  \(\Delta\) is a reflexive polytope, then \(\Delta^{\vee}\) is also a lattice
  polytope and satisfies \((\Delta^{\vee})^{\vee}=\Delta\). 
  The normal fan of \(\Delta\)
  is the face fan of \(\Delta^{\vee}\) and vice versa.
  \item For a reflexive polytope \(\Delta\), a \emph{nef-partition} on
  \(\mathbf{P}_{\Delta}\) is a decomposition
  \begin{equation*}
  \Sigma_{\Delta}(1) = I_{1} \dot\cup\cdots\dot\cup I_{r}
  \end{equation*}
  such that \(E_{s}:=\sum_{\rho\in I_{s}} D_{\rho}\)
  is numerical effective for each \(s\).
  This gives rise to a Minkowski sum decomposition
  \(\Delta=\Delta_{1}+\cdots+\Delta_{r}\), where \(\Delta_{i}=\Delta_{E_{i}}\)
  is the polytope of \(E_{i}\).
  \end{itemize}
We recall Batyrev--Borisov's duality construction.
Let \(\Delta\) be a reflexive polytope and 
\(\Sigma_{\Delta}(1) = I_{1} \dot\cup\cdots\dot\cup I_{r}\)
be a nef-partition.
We define \(\nabla_{k}:=\mathrm{Conv}(
\{\mathbf{0}\}\cup I_{k})\) and 
\(\nabla:=\nabla_1+\ldots+\nabla_{r}\).
It turns out that \(\nabla\) is a reflexive polytope in \(N_{\mathbb{R}}\)
whose polar polytope is 
\(\nabla^{\vee}=\mathrm{Conv}(\Delta_{1},\ldots,\Delta_{r})\)
and \(\nabla_1+\cdots+\nabla_{r}\) corresponds to a 
nef-partition on \(\mathbf{P}_{\nabla}\),
called the \emph{dual nef-partition}.
We denote by \(F_{1},\ldots,F_{r}\) 
the corresponding nef toric divisors on \(\mathbf{P}_{\nabla}\).
Then the polytope of \(F_{j}\) is \(\nabla_{j}\).

\subsection{A construction of singular double covers}
\label{subsection:singular-double-covers}
In this subsection,
we recall the construction of Calabi--Yau double covers
introduced in \cite{2020-Hosono-Lee-Lian-Yau-mirror-symmetry-for-double-cover-calabi-yau-varieties}.
Let \(\Delta=\Delta_{1}+\cdots+\Delta_{r}\) be a nef-partition
and \(\nabla = \nabla_{1}+\cdots+\nabla_{r}\)
be the dual nef-partition.
Let \(X\to \mathbf{P}_{\Delta}\) and \(X^{\vee}\to \mathbf{P}_{\nabla}\)
be maximal projective crepant partial desingularizations
(MPCP desingularizations for short hereafter)
for \(\mathbf{P}_{\Delta}\) and \(\mathbf{P}_{\nabla}\)
\cite{1994-Batyrev-dual-polyhedra-and-mirror-symmetry-for-calabi-yau-hypersurfaces-in-toric-varieties}.
Note that MPCP desingularizations are smooth
and not unique in general. 
According to the construction, 
the polytopes \(\Delta_{i}\) and \(\nabla_{j}\)
correspond to \(E_{i}\) on \(\mathbf{P}_{\Delta}\) and 
\(F_{j}\) on \(\mathbf{P}_{\nabla}\), respectively.
The nef-partitions on 
\(\mathbf{P}_{\Delta}\) and \(\mathbf{P}_{\nabla}\)
give rise to 
nef-partitions on \(X\) and \(X^{\vee}\) via pullback.
To save the notation, the corresponding toric divisors
on \(X\) and \(X^{\vee}\) and their polytopes will be still 
denoted by \(E_{i}\), \(F_{j}\) and
\(\Delta_{i}\), \(\nabla_{j}\).

\begin{hypothesis}
Throughout this paper, unless otherwise stated, 
we shall assume that 
\begin{center}
{\it \(X\) and \(X^{\vee}\) are both smooth}.
\end{center}
Equivalently, we assume that both \(\Delta\) and \(\nabla\) admit a uni-modular triangulation.
\end{hypothesis}

Using these data, we can construct a pair of 
families of Calabi--Yau varieties.
\begin{definition}
Let \(\Delta=\Delta_{1}+\cdots+\Delta_{r}\) be a decomposition
representing a nef-partition \(E_{1}+\cdots+E_{r}\) on \(X\).
A \emph{gauge fixed double cover 
branched along the nef-partition
\(\Delta_{1}+\cdots+\Delta_{r}\)} over \(X\) is
the double cover \(Y\to X\) constructed from the 
section 
\begin{equation*}
\textstyle
s=\prod_{i=1}^{r}\prod_{j=1}^{2}s_{i,j}
\end{equation*} 
where
\(s_{i,1}\) is the global section of \(E_{i}\) corresponding 
to the lattice point \(\mathbf{0}\in \Delta_{i}\) and 
\(s_{i,2}\) is a smooth global section of \(E_{i}\)
such that \(\mathrm{div}(s)\)
is a simple normal crossing (SNC) divisor.
\end{definition}

In the present situation, one can
prove that \(Y\) is Cohen--Macaulay. Moreover, 
since the branch locus is a SNC divisor,
\(Y\) is a normal variety
with at worst quotient singularities. Also, the fact that the branch 
divisor is linearly equivalent to \(-2K_{X}\)
implies that \(Y\) is Calabi--Yau; namely
\(\omega_{Y}\cong\mathscr{O}_{Y}\).

Deforming the sections \(s_{i,2}\) yields 
a singular Calabi--Yau family \(\mathcal{Y}\to V\).
Here 
\begin{equation*}
V\subset \mathrm{H}^{0}(X,E_{1})\times\cdots\times\mathrm{H}^{0}(X,E_{r})
\end{equation*}
is an open subset where
the branch divisor \(\mathrm{div}(s)\)
is SNC. Likewise, we can apply the construction
to the dual nef-partition \(\nabla=\nabla_{1}+\cdots+\nabla_{r}\) and obtain
another singular Calabi--Yau family
\(\mathcal{Y}^{\vee}\to U\),
where 
\begin{equation*}
U\subset \mathrm{H}^{0}(X^{\vee},F_{1})
\times\cdots\times\mathrm{H}^{0}(X^{\vee},F_{r})
\end{equation*}
is an open subset such that
the branch divisor is SNC.

\begin{instance}
\label{ex:p2}
Let \(\Delta=\mathrm{Conv}\{(2,-1),(-1,2),(-1,-1)\}\). 
We then have \(X=\mathbf{P}_{\Delta} = \mathbf{P}^{2}\)
since \(\mathbf{P}_{\Delta}\)
is already smooth.
Consider the trivial nef-partition \(\Delta=\Delta_{1}\).
The dual nef-partition is given by 
\begin{equation*}
\nabla = \nabla_{1} = 
\Delta^{\vee} = \mathrm{Conv}\{(1,0),(0,1),(-1,-1)\}.
\end{equation*}
The toric variety \(\mathbf{P}_{\nabla}\)
is the mirror \(\mathbf{P}^{2}\)
whose MPCP desingularization \(X^{\vee}\to\mathbf{P}_{\nabla}\) is smooth.
Applying the construction to the present case,
we obtain two singular Calabi--Yau families
\(\mathcal{Y}\to V\) and \(\mathcal{Y}^{\vee}\to U\).
The fiber of \(\mathcal{Y}\to V\)
is a double cover over \(X\) branched along three lines and 
one cubic whereas the fiber of \(\mathcal{Y}^{\vee}\to U\)
is a double cover over \(X^{\vee}\) branched along 
a wheel consisting of nine \(\mathbf{P}^{1}\)s and one 
smooth elliptic curve.
\end{instance}

\begin{instance}
\label{ex:p3}
Let \(\Delta = \mathrm{Conv}\{(-3,1,1),(1,-3,1),(1,1,-3),(-1,-1,-1)\}\).
In the present case, we have \(X = \mathbf{P}_{\Delta} = \mathbf{P}^{3}\).
Consider the trivial nef-partition \(\Delta = \Delta_{1}\). The 
dual nef-partition is given by
\begin{equation*}
\nabla = \nabla_{1} = \Delta^{\vee}
=\mathrm{Conv}\{(1,0,0),(0,1,0),(0,0,1),(-1,-1,-1)\}.
\end{equation*}
Note that in this case, \(\mathbf{P}_{\nabla}\) admits
a smooth MPCP desingularization \(X^{\vee}\to\mathbf{P}_{\nabla}\). (See also
\cite{2021-Lian-Zhu-on-the-hyperplane-conjecture-for-periods-of-
calabi-yau-hypersurfaces-in-pn}*{Appendix} for details.)
\end{instance}

\begin{conjecture}
Let \(Y\) and \(Y^{\vee}\) be the fiber of 
\(\mathcal{Y}\to V\) and \(\mathcal{Y}^{\vee}\to U\) respectively. 
Then \((Y,Y^{\vee})\) is a mirror pair.
\end{conjecture}

\section{Mirror symmetry: the topological test}

In this section, we compute the topological Euler characteristic
of the double cover Calabi--Yau varieties as well as their Hodge numbers. 
For an \(n\)-dimensional (quasi-projective) 
variety \(W\), the topological Euler characteristic 
is defined to be the alternating sum
\begin{equation*}
\chi(W) := \sum_{k=0}^{2n} (-1)^{k}\dim \mathrm{H}^{k}(W)
=\sum_{k=0}^{2n} (-1)^{k}\dim \mathrm{H}_{c}^{k}(W).
\end{equation*}
Let \(\pi\colon Y\to X\) be a double cover branched along \(D\).
Then \(Y\setminus \pi^{-1}(D)\to X\setminus D\) is a finite \'{e}tale cover of degree \(2\).
We then have 
\begin{align}
\label{equation:euler-charactersitic-branched-covers}
\begin{split}
\chi(Y) &= \chi(\pi^{-1}(D)) + \chi(Y\setminus \pi^{-1}(D))\\ 
&= \chi(D) + 2\cdot\chi(X\setminus D)\\
&= \chi(D) + 2(\chi(X)-\chi(D)).
\end{split}
\end{align}

To compute the topological Euler characteristic
of the branch divisor, we need a result by
Danilov and Khovanskii. To this end, let us intorduce some
terminologies and notation.
Let \(Z_{1},\ldots,Z_{r}\) be nef torus 
invariant divisors on a projective toric manifold 
\(X\) and \(\Delta_{Z_{i}}\)
be the polytope of \(Z_{i}\).
Put
\begin{equation*}
\Delta_{Z_{1}}\star\cdots\star\Delta_{Z_{r}}:=
\mathrm{Conv}(\{\mathrm{e}_{1}\}\times\Delta_{Z_{1}},\ldots,
\{\mathrm{e}_{r}\}\times\Delta_{Z_{r}})
\end{equation*} 
in \(\mathbb{R}^{r}\times M_{\mathbb{R}}\).
Here \(\{\mathrm{e}_{1},\ldots,\mathrm{e}_{r}\}\)
is the standard basis of \(\mathbb{R}^{r}\).
This is called the \emph{Cayley polytope} of \(\Delta_{Z_{1}},\ldots,
\Delta_{Z_{r}}\).
For a nonempty subset \(I\subset\{1,\ldots,r\}\), we similarly put
\begin{equation*}
\Delta^{\star I}:=\star_{i\in I} \Delta_{Z_{i}}\subset 
\mathbb{R}^{|I|}\times M_{\mathbb{R}}.
\end{equation*}
Let \(\Lambda_{I}\) be 
the pyramids with vertex \(\mathbf{0}\) and 
base \(\Delta^{\star I}\) 
in \(\mathbb{R}^{|I|}\times M_{\mathbb{R}}\).
Now we can state the result by Danilov and Khovanskii.

\begin{theorem}[cf.~\cite{1986-Danilov-Khovanskii-newton-polyhedra-and-an-algorithm-for-calculating-hodge-deligne-numbers}*{\S6}]
\label{theorem:DK-euler-characteristic}
For general \(D_{i}\) in the linear system \(|Z_{i}|\), we have
\begin{equation*}
\chi(D_{1}\cap\cdots\cap D_{r}\cap T) =
-\sum_{I} (-1)^{n+|I|-1}\mathrm{vol}_{n+|I|}(\Lambda_{I}),
\end{equation*}
where the summation runs over all nonempty subsets \(I\subset\{1,\ldots,k\}\)
and \(\mathrm{vol}_{d}\) is the normalized volume in \(d\)-dimensional spaces.
\end{theorem}

Applying the theorem to the case \(Z_{i} = E_{i}\) and \(X
\to\mathbf{P}_{\Delta}\) is a MPCP desingularization, we can
compute the topological Euler characteristic of
the branch divisor which turns out to be
\begin{align*}
\chi(D)&=\chi(X) - \chi(T\cap(D_{1}\cup\cdots\cup D_{r}))\\
&=\chi(X) + (-1)^{n} \mathrm{vol}_{n+r}(\Lambda)
\end{align*}
where \(D_{i}\) is a general element in \(|E_{i}|\)
and \(\Lambda = \Lambda_{\{1,\ldots,r\}}\).
When \(X^{\vee}\) is smooth,
together with the fact that \(F_{i}\) are nef, one can show that
\(\mathrm{vol}_{n+r}(\Lambda)=\chi(X^{\vee})\).
Denote by \(Y\) (resp.~\(Y^{\vee}\)) a fiber of 
\(\mathcal{Y}\to V\) (resp.~\(\mathcal{Y}^{\vee}\to U\)).
It follows that 
\begin{equation*}
\chi(Y) = \chi(X) + (-1)^{n}\chi(X^{\vee})
\end{equation*}
and therefore \(\chi(Y)=(-1)^{n}\chi(Y^{\vee})\).
To summarize, we obtain
\begin{theorem}[cf.~\cite{2020-Hosono-Lee-Lian-Yau-mirror-symmetry-for-double-cover-calabi-yau-varieties}*{Theorem 2.2}]
Let notation be as above. We have 
\begin{equation*}
\chi(Y)=(-1)^{n}\chi(Y^{\vee}).
\end{equation*}
\end{theorem}

We can further determine the Hodge numbers \(h^{p,q}(Y)\) for \(p+q\ne n\).
As observed by Baily 
\cite{1957-Baily-on-the-imbedding-of-v-manifolds-in-projective-space} and Steenbrink \cite{1977-Steenbrink-mixed-hodge-structure-on-the-vanishing-cohomology},
most statements in Hodge theory generalize to orbifolds.
Combined with the affine vanishing theorem,
we can prove the following proposition.
\begin{proposition}
\label{prop:hodge-number}
Let notation be as above and \(\pi\colon Y\to X\)
be the branched double cover. We have
\(h^{p,q}(X) = h^{p,q}(Y)\) for \(p+q\ne n\).
\end{proposition}
For more details, we refer the reader to 
\cites{2020-Hosono-Lee-Lian-Yau-mirror-symmetry-for-double-cover-calabi-yau-varieties,2012-Arapura-hodge-theory-of-cyclic-covers-branched-over-a-union-of-hyperplanes}.
\begin{corollary}
When \(n=3\), we have \(h^{p,q}(Y)=h^{3-p,q}(Y^{\vee})\) for all \(p,q\).
\end{corollary}

\begin{instance}[Example \ref{ex:p2} continued]
\label{ex:p2-top}
Let us retain the notation in Example \ref{ex:p2}.
On one hand, \(Y\) is a singular \(K3\) surface with \(12\) 
ODPs singularities. A direct computation shows that
\(\chi(Y) = 24-12 = 12\).

On the other hand, \(Y^{\vee}\) is also a singular \(K3\) surface
with at worst ODPs singularities. Let us now compute its topological
Euler characteristic. \(Y^{\vee}\) is
singular over the points on \(X^{\vee}\) where
the branch divisor is singular. A generic 
section in \(\mathrm{H}^{0}(X^{\vee},-K_{X^{\vee}})\)
only touches the Weil divisors associated with
\((-1,-1)\), \((2,-1)\), and \((-1,2)\) with intersection number \(1\).
Together with the wheel of \(\mathbf{P}^{1}s\), we obtain
\(9+3\) ODPs singularities. We conclude that
\begin{equation*}
\chi(Y^{\vee}) = 24 - 12 = 12 = \chi(Y).
\end{equation*}
One can also check that 
\(\mathrm{vol}_{4}(\Lambda_{1})=3\). 
(See the paragraph right before Theorem
\ref{theorem:DK-euler-characteristic} for notation.)

\end{instance}

\begin{instance}
\label{ex:p3-top}
Retain the notation in Example \ref{ex:p3}.
Let us compute that Euler characteristic of \(Y\) directly.
Denote by \(D\) the scheme-theoretic zero of 
a generic section in \(\mathrm{H}^{0}(X,-K_{X})\)
and by \(D_{x}\), \(D_{y}\), \(D_{z}\), \(D_{w}\)
the coordinate hyperplanes.

One can easily compute
\begin{align*}
\chi(D_{a}\cap D_{b})&=2,\\
\chi(D_{a}\cap D_{b}\cap D_{c})&=1,\\
\chi(D\cap D_{a}) &=-4,\\
\chi(D\cap D_{a}\cap D_{b})&=4,
\end{align*}
where \(a,b,c\in\{x,y,z,w\}\) are distinct elements.
It follows that \(\chi(Y)=-60\). Moreover, 
by Proposition \ref{prop:hodge-number}, 
\begin{align*}
h^{1,1}(Y)=h^{1,1}(X)=1,\\
h^{1,0}(Y)=h^{0,1}(Y)=h^{2,0}(Y)=h^{0,2}(Y)=0.
\end{align*}
This implies that \(h^{2,1}(Y)=31\).

On the other hand, since \(\chi(X^{\vee})=\mathrm{vol}_{3}(\Delta)=
64\) and \(X^{\vee}\)
is smooth and toric, we have \(h^{1,1}(X^{\vee})=(64-2)/2=31\).
By Proposition \ref{prop:hodge-number}, \(h^{1,1}(Y^{\vee})=
h^{1,1}(X^{\vee})=31\). 
Together with \(\chi(Y^{\vee})=60\), we see that 
\(h^{2,1}(Y^{\vee})=1\).
\end{instance}

\section{Mirror symmetry: the quantum test}
The purpose of this section is to explain the quantum test.
Let \(\mathcal{Y}^{\vee}\to U\)
be the singular Calabi--Yau family constructed in 
\S\ref{subsection:singular-double-covers}.
Let \(\Sigma\) be the fan defining \(X\).
By its very construction, we see that
\begin{equation*}
\mathrm{H}^{0}(X^{\vee},F_{i}) \cong \bigoplus_{\rho\in \nabla_{i}\cap N} 
\mathbb{C}\cdot t^{\rho}.
\end{equation*}
and that the set \(I_{k}\) in the nef-partition 
\(\Sigma(1)=I_{1}\dot\cup\cdots\dot\cup I_{r}\)
is identified with 
\begin{equation*}
\{\rho\in \Sigma(1)~|~\mathbf{0}\ne \rho\in \nabla_{k}\cap N\}.
\end{equation*}
Writing \(I_{k}=\{\rho_{k,1},\ldots,\rho_{k,n_{k}}\}\) with
\(n_{k} = \# I_{k}\), we put
\(\nu_{i,j}:=(\rho_{i,j},\delta_{1,i},\ldots,\delta_{r,i})\) 
and additionally
\(\nu_{i,0}:=(\mathbf{0},\delta_{1,i},\ldots,\delta_{r,i})\),
where \(\delta_{i,j}\) is the Kronecker delta.
Let 
\begin{equation*}
A=
\begin{bmatrix}
\nu_{1,0}^{\intercal} & \cdots & \nu_{r,n_{r}}^{\intercal}
\end{bmatrix}\in\mathrm{Mat}_{(n+r)\times(p+r)}(\mathbb{Z})~
\mbox{where}~
p=n_{1}+\cdots+n_{r}.
\end{equation*}
It can be shown that the period integrals
for \(\mathcal{Y}^{\vee}\to U\) are governed
by the GKZ \(A\)-hypergeometric system associated 
with the matrix \(A\) and 
a \emph{fractional} exponent 
\begin{equation*}
\beta = 
\begin{bmatrix}\mathbf{0} & -1/2 & \cdots & -1/2
\end{bmatrix}^{\intercal}
\in\mathbb{Q}^{n+r}.
\end{equation*}
Let \(Y^{\vee}\) be a reference fiber 
in \(\mathcal{Y}^{\vee}\to U\) and 
\(D\) be the branch divisor of 
the cover \(Y^{\vee}\to X^{\vee}\). 
Then \(\pi\colon Y^{\vee}\setminus D\to X^{\vee}\setminus D\)
is an \'{e}tale double cover. Denote by \(\mathscr{L}\)
the unique non-trivial eigensheaf in 
\(\pi_{\ast}\mathbb{C}_{Y^{\vee}\setminus D}\).
The period integrals over \(X^{\vee}\setminus D\)
are of the form
\begin{equation}
\label{eq:general-period}
\Pi_{\gamma}(x):=\int_{\gamma} 
\frac{1}{u_{1,2}^{1/2}\cdots u_{r,2}^{1/2}}
\frac{\mathrm{d}t_{1}}{t_{1}}\wedge\cdots
\wedge\frac{\mathrm{d}t_{n}}{t_{n}},
\end{equation}
where \(\gamma\) is an \(n\)-cycle 
in \(X^{\vee}\setminus D\) with coefficient in 
\(\mathscr{L}\)
and 
\begin{equation*}
u_{i,2} = x_{i,0}+\sum_{j=1}^{n_{i}} x_{i,j}t^{\rho_{i,j}}\in 
\mathrm{H}^{0}(X^{\vee},F_{i})
\end{equation*}
is the universal section. 
We shall call \(\Pi_{\gamma}(x)\)
an \emph{affine period integral}.
It is easy to check that \(\Pi_{\gamma}(x)\)
is annihilated by the relevant GKZ \(A\)-hypergeometric system.
For convenience, we also define 
the \emph{normalized affine period integrals}
to be
\begin{equation*}
\bar{\Pi}_{\gamma}(x):=\left(\prod_{i=1}^{r} x_{i,0}\right)^{1/2}
\Pi_{\gamma}(x).
\end{equation*}

Since \(\bar{\Pi}_{\gamma}(x)\)
is \((\mathbb{C}^{\ast})^{r+n}\)-invariant,
it descends to a local section of a locally constant sheaf on 
a suitable open subset in \((\mathbb{C}^{\ast})^{r+p}
\slash (\mathbb{C}^{\ast})^{r+n}\). 
Regarding \(A\) as a map \(\mathbb{Z}^{r+p}\to\mathbb{Z}^{r+n}\),
we have an identification
\begin{equation*}
(\mathbb{C}^{\ast})^{r+p}
\slash (\mathbb{C}^{\ast})^{r+n}\cong 
\mathrm{Hom}_{\mathbb{Z}}(L,\mathbb{C}^{\ast}),
\end{equation*}
where \(L=\mathrm{ker}(A)\). Any
complete fan \(F\) in \(L^{\vee}\otimes{\mathbb{R}}\)
determines a toric compactification \(X_{F}\).
Recall that a smooth boundary point \(p\in X_{F}\)
is called a maximal degeneracy point if near \(p\)
there exists exactly one normalized affine period integral which
extends to \(p\) holomorphically.
Following the ideas
in \cite{1997-Hosono-Lian-Yau-maximal-degeneracy-points-of-gkz-systems},
the algebraic torus \((\mathbb{C}^{\ast})^{r+p}
\slash (\mathbb{C}^{\ast})^{r+n}\) admits a toric compactification
via the \emph{secondary fan} \(S\Sigma\)
associated to \(A\). Moreover, using the toric ideal associated to 
\(A\), we can define another fan \(G\Sigma\), called
the Gr\"{o}bner fan, which refines \(S\Sigma\).
In the present case, we can prove that
\begin{theorem}
\label{thm:main-theorem-1}
For every toric resolution \(X_{G\Sigma'}\to X_{G\Sigma}\), there exists
at least one maximal degeneracy point in \(X_{G\Sigma'}\). 
\end{theorem}

We can write down the unique holomorphic normalized 
period near \(p\) explicitly.
Pick an \(\alpha\in\mathbb{C}^{p+r}\) such that 
\(A(\alpha)=\beta\).
In the present case, we can pick
\(\alpha=(\alpha_{i,j})\) with
\(\alpha_{i,0}=-1/2\) for \(i=1,\ldots,r\) 
and \(\alpha_{i,j}=0\) for \(i=1,\ldots,r\) and \(j=1,\ldots,n_{i}\).
One solution to the GKZ system is given by
\begin{equation}
\label{eq:hol-series-sol-1}
\sum_{\ell\in L} 
\frac{\prod_{i=1}^{r}\Gamma(-\alpha_{i,0}-\ell_{i,0})}
{\prod_{i=1}^{r}\Gamma(-\alpha_{i,0})
\prod_{i=1}^{r}\prod_{j=1}^{n_i} \Gamma(\ell_{i,j}+\alpha_{i,j}+1)}
(-1)^{\sum_{i} \ell_{i,0}}x^{\ell+\alpha}.
\end{equation}
Here the components of \(\ell\in 
L\subset \mathbb{Z}^{p+r}\) are labeled by \((i,j)\)
with \(1\le i\le r\) and \(0\le j\le n_{i}\) and 
\(x_{i,j}\) are the 
coordinates for our GKZ \(A\)-hypergeometric system.

\subsection[]{Generalized Frobenius method}
Now let us explain the generalized Frobenius method.
Let \(D_{i,j}\) be the toric divisor associated with \(\rho_{i,j}\). 
Combining \eqref{eq:hol-series-sol-1}
with these cohomology classes, we define a cohomology-valued power series
\begin{equation}
\label{eq:coh-valued-series}
B_{X}^{\alpha}(x):=\left(\sum_{\ell\in \overline{\mathrm{NE}}(X)\cap L} 
\mathcal{O}_{\ell}^{\alpha} x^{\ell+\alpha}\right)
\exp\left(\sum_{i=1}^{r}\sum_{j=0}^{n_i}(\log x_{i,j}) D_{i,j}\right),
\end{equation}
where \(\overline{\mathrm{NE}}(X)\) is the Mori cone of \(X\) and 
\begin{equation*}
\mathcal{O}^\alpha_\ell
:=\frac{\prod_{i=1}^{r}(-1)^{\ell_{i,0}}
\Gamma(-D_{i,0}-\ell_{i,0}-\alpha_{i,0})}{\prod_{i=1}^{r}\Gamma(-\alpha_{i,0})
\prod_{i=1}^r\prod_{j=1}^{n_i}\Gamma(D_{i,j}+\ell_{i,j}+\alpha_{i,j}+1)}.
\end{equation*}
with \(D_{i,0}:=-\sum_{j=1}^{n_{i}} D_{i,j}\). 

The cohomology-valued series \eqref{eq:coh-valued-series} 
was introduced by Hosono et~al.~in
\cite{1996-Hosono-Lian-Yau-gkz-generalized-hypergeometric-systems-in-mirror-symmetry-of-calabi-yau-hypersurfaces} 
which encodes the information from the \(A\) model and the \(B\) model for 
a Calabi--Yau mirror pair. The series is 
also called a Givental's \(I\)-function 
in the literature.
We have the following theorem.
\begin{theorem}
\label{prop:coh-ser-solu-gkz-intro}
The pairings \(\langle B^{\alpha}_{X}(x),h\rangle\)
give a complete set of solution to the GKZ \(A\)-hypergeometric system associated 
with \(\mathcal{Y}^{\vee}\to U\)
when \(h\in\mathrm{H}^{\bullet}(X,\mathbb{C})^{\vee}\)
runs through a basis of \(\mathrm{H}^{\bullet}(X,\mathbb{C})^{\vee}\).
\end{theorem}

\begin{instance}
[Example \ref{ex:p3-top} continued]
Let us keep the notation.
We describe the GKZ \(A\)-hypergeometric system
for the singular family \(\mathcal{Y}^{\vee}\to U\).
Note that 
\begin{equation*}
\mathrm{H}^{0}(X^{\vee},-K_{X^{\vee}})\cong \bigoplus_{\rho\in \nabla\cap N}
\mathbb{C}\cdot t^{\rho}.
\end{equation*}
The period integrals of \(\mathcal{Y}^{\vee}\to U\)
are governed by the GKZ \(A\)-hypergeometric system with
\begin{equation*}
A = 
\begin{bmatrix}
1 & 1 & 1 & 1 & 1\\
0 & 1 & 0 & 0 &-1\\
0 & 0 & 1 & 0 &-1\\
0 & 0 & 0 & 1 &-1
\end{bmatrix}~\mbox{and}~
\beta=
\begin{bmatrix}
-1/2\\
0\\
0\\
0
\end{bmatrix}.
\end{equation*}
Regarding \(A\) as a linear map \(A\colon
\mathbb{Z}^{5}\to\mathbb{Z}^{4}\) as before, 
we may take
\begin{equation*}
\alpha=
\begin{bmatrix}
-1/2\\
0\\
0\\
0\\
0
\end{bmatrix}
\end{equation*}
so that \(A(\alpha)=\beta\). We can compute
the lattice relation of \(A\)
\begin{equation*}
\mathrm{ker}(A)=\mathbb{Z}\langle(-4,1,1,1,1)
\rangle.
\end{equation*}
We also note that \(\mathrm{vol}_{4}(A)=4\).
Denote by \(x_{0},x_{1},\ldots,x_{4}\) the
variable attached to the columns of \(A\) in the GKZ system.
We can write down the unique holomorphic series solution
\begin{equation*}
\sum_{n\ge 0}\frac{\Gamma(1/2+4n)}{\Gamma(1/2)\Gamma(n+1)^{4}}
\left(\frac{x_{1}x_{2}x_{3}x_{4}}{x_{0}^{4}}\right)^{n}.
\end{equation*}

In the present case, we have \(r=j=1\). 
For simplicity, we will drop the index \(i\) 
in the above formula and 
write \(D_{j}\equiv D_{1,j}\).
The cohomology-valued series 
\eqref{eq:coh-valued-series} becomes
\begin{equation}
\label{eq:coh-series-ex-p3}
B_{X}^{\alpha}(x)=\left(\sum_{n\ge 0} 
\mathcal{O}_{n}^{\alpha} \left(\frac{x_{1}x_{2}x_{3}x_{4}}{x_{0}^{4}}
\right)^{n}\frac{1}{x_{0}^{1/2}}\right)
\exp\left(\sum_{j=0}^{4}(\log x_{j}) D_{j}\right)
\end{equation}
where
\begin{equation*}
\mathcal{O}^{\alpha}_{n}
=\frac{\Gamma(-D_{0}+4n+1/2)}{\Gamma(1/2)
\prod_{j=1}^{4}\Gamma(D_{j}+n+1)}.
\end{equation*}
Denote by \(H\) the hyperplane class of \(\mathbf{P}^{3}\).
We have \(D_{1}=D_{2}=D_{3}=D_{4}=H\) and \(D_{0}=-4H\).
We may re-write \eqref{eq:coh-series-ex-p3} into
\begin{equation}
\label{eq:coh-series-1}
B_{X}^{\alpha}(x)=\frac{1}{x_{0}^{1/2}}\left(\sum_{n\ge 0} 
\mathcal{O}_{n}^{\alpha} 
\left(\frac{x_{1}x_{2}x_{3}x_{4}}{x_{0}^{4}}\right)^{n}\right)
\exp\left(\log 
\left(\frac{x_{1}x_{2}x_{3}x_{4}}{x_{0}^{4}}\right)\cdot H\right)
\end{equation}
with 
\begin{equation*}
\mathcal{O}^{\alpha}_{n}
=\frac{\Gamma(4H+4n+1/2)}{\Gamma(1/2)
\Gamma(H+n+1)^{4}}.
\end{equation*}
If we expand the series \eqref{eq:coh-series-1} 
according to the cohomology basis 
\(\{\mathbf{1},H,H^{2},H^{3}\}\), then
the coefficients give a complete set of the
solutions to the relevant GKZ \(A\)-hypergeometric
system.

\end{instance}

\subsection{The quantum test}
In this subsection, we will 
consider the double cover over \(\mathbf{P}^{3}\)
branched along four hyperplanes and one quartic
(see Example \ref{ex:p3})
and carry out the quantum test in this case.
Let us keep the notation in Example \ref{ex:p3}
and Example \ref{ex:p3-top}.

\subsubsection{Crepant resolutions} 
According to the results in 
\cite{2013-Sheng-Xu-Zuo-maximal-families-of-calabi-yau-manifolds-with-minimal-length-length-yukawa-coupling}, there exists a
family of crepant
resolutions \(\tilde{\mathcal{Y}}^{\vee}\to \mathcal{Y}^{\vee}\to U\)
without modifying \(h^{p,q}\) of each fiber for \(p\ne q\).
Fix a reference fiber \(Y^{\vee}\)
of \(\mathcal{Y}^{\vee}\to U\) as before 
and denote by \(\tilde{Y}^{\vee}\to Y^{\vee}\) 
the crepant resolution.
One can easily see that 
\(\tilde{Y}^{\vee}\) is smooth Calabi--Yau
with \(h^{2,1}(\tilde{Y}^{\vee})=h^{2,1}(Y^{\vee})=1\). 
Moreover, we have \(\mathrm{H}^{3}(\tilde{Y}^{\vee},\mathbb{Q})
\cong\mathrm{H}^{3}(Y^{\vee},\mathbb{Q})\).

\subsubsection{Picard--Fuchs equations}
Recall that the GKZ \(A\)-hypergeometric system
for the family \(\mathcal{Y}^{\vee}\to U\)
is given by the data
\begin{equation*}
A = 
\begin{bmatrix}
1 & 1 & 1 & 1 & 1\\
0 & 1 & 0 & 0 &-1\\
0 & 0 & 1 & 0 &-1\\
0 & 0 & 0 & 1 &-1
\end{bmatrix}~\mbox{and}~
\beta=
\begin{bmatrix}
-1/2\\
0\\
0\\
0
\end{bmatrix}.
\end{equation*}
Let us carry out the Picard--Fuchs equations
on the one-dimensional moduli.
Set \(\ell := (-4,1,1,1,1)\in\mathrm{ker}(A)\). It is a generator 
of \(\mathrm{ker}(A)\)
and define
\begin{equation*}
z = x^{\ell} := \frac{x_{1}x_{2}x_{3}x_{4}}{x_{0}^{4}}.
\end{equation*}
Here \(z\) can be regarded as the coordinate on the quotient 
\((\mathbb{C}^{\ast})^{5}\slash \mathrm{ker}(A)\otimes\mathbb{C}^{\ast}
\cong\mathbb{C}^{\ast}\).
Consider the box operator
\begin{equation*}
\Box_{\ell} = \partial_{1}\partial_{2}
\partial_{3}\partial_{4}-\partial_{0}^{4},~\mbox{where}~
\partial_{j}:=\frac{\partial}{\partial x_{j}}.
\end{equation*}
To carry out the Picard--Fuchs equation,
we consider the conjugate operator
\begin{align}
x^{-\alpha}x_{1}x_{2}x_{3}x_{4}\Box_{\ell}x^{\alpha}&=
x_{0}^{1/2}x_{1}x_{2}x_{3}x_{4}\Box_{\ell}x_{0}^{-1/2}\notag\\
&=x_{1}x_{2}x_{3}x_{4}\partial_{1}\partial_{2}
\partial_{3}\partial_{4}-zx_{0}^{1/2}x_{0}^{4}
\partial_{0}^{4}x_{0}^{-1/2}\notag\\
&=\theta_{1}\theta_{2}\theta_{3}\theta_{4}- z
\left(\theta_{0}-\frac{7}{2}\right)
\left(\theta_{0}-\frac{5}{2}\right)
\left(\theta_{0}-\frac{3}{2}\right)
\left(\theta_{0}-\frac{1}{2}\right)\notag\\
&=\theta_{z}^{4} - 256z
\left(\theta_{z}+\frac{7}{8}\right)
\left(\theta_{z}+\frac{5}{8}\right)
\left(\theta_{z}+\frac{3}{8}\right)
\left(\theta_{z}+\frac{1}{8}\right)\label{eq:pf-eq}.
\end{align}
Here \(\theta_{i}=x_{i}\partial_{i}\) and \(\theta_{z}=z\partial_{z}\).

\begin{remark}
We note that the equation \eqref{eq:pf-eq} is the same as
the Picard--Fuchs equation for 
the mirror of a smooth degree \(8\)
hypersurface in the 
weighted projective space \(\mathbf{P}(1,1,1,1,4)\).
\end{remark}

\begin{remark}
From Batyrev's duality construction, 
\(X^{\vee}\) is a resolution of the toric
variety \(\mathbf{P}_{\nabla}\cong \mathbf{P}^{3}\slash G\)
for some finite abelian group \(G\). Indeed,
\begin{equation*}
\nabla = \mathrm{Conv}\{(1,0,0),(0,1,0),(0,0,1),(-1,-1,-1)\}.
\end{equation*}
The normal fan of \(\nabla\) is the face fan of
its dual 
\begin{equation*}
\nabla^{\vee}=\Delta=
\mathrm{Conv}\{(3,-1,-1),(-1,3,-1),(-1,-1,3),(-1,-1,-1)\}.
\end{equation*}
The group \(G\) can be identified with the quotient group
\(\mathbb{Z}^{3}\slash N_{0}\) where
\(N_{0}\) is the sublattice generated by 
\((3,-1,-1),(-1,3,-1),(-1,-1,3),(-1,-1,-1)\).

The singular mirror \(Y^{\vee}\) is a double cover
branched along the union of toric divisors on \(X^{\vee}\)
together with a smooth anti-canonical section.
The universal family for \(\tilde{Y}^{\vee}\) 
can be realized as a partial resolution of
the double cover over \(\mathbf{P}_{\nabla}\)
branched along the union of toric divisors and
\begin{equation*}
\{z_{1}^{4}+z_{2}^{4}+z_{3}^{4}+z_{4}^{4}+x_{0}\cdot 
z_{1}z_{2}z_{3}z_{4}=0\}
\slash G,~~\mbox{where}~
[z_{1}\mathpunct{:}z_{2}\mathpunct{:}z_{3}\mathpunct{:}z_{4}]\in
\mathbf{P}^{3}.
\end{equation*}
The variable \(z:=x_{0}^{-4}\) is the coordinate on
the moduli space of \(\tilde{Y}^{\vee}\) (or \(Y^{\vee}\)).
\end{remark}

\subsubsection{The instanton prediction}
Mimicking the classical case, we can compute
the unique holomorphic period around 
\(z=0\)
\begin{equation}
\label{eq:hol-series-sol}
\omega_{0}(z):=\sum_{n\ge 0} 
\frac{\Gamma(4n+1/2)}{\Gamma(1/2)\Gamma(n+1)^{4}}
z^{n}.
\end{equation}
Denote by \(\Omega(z)\) a relative holomorphic
top form on the moduli space of \(\tilde{Y}^{\vee}\).
Let us compute the Yukawa coupling
\begin{equation}
\label{equation:yukawa-coupling}
\left\langle \theta_{z},\theta_{z},\theta_{z}\right\rangle^{\Omega}
:=\int_{\tilde{Y}^{\vee}} \Omega(z)\wedge \theta_{z}^{3}\Omega(z).
\end{equation}
By Griffiths transversality, 
\begin{equation*}
\int_{\tilde{Y}^{\vee}} \Omega(z)\wedge \theta_{z}^{2}\Omega_{z}=0.
\end{equation*}
Differentiating the displayed equation twice yields
\begin{equation*}
\int_{\tilde{Y}^{\vee}} \theta_{z}\Omega(z)\wedge \theta_{z}^{3}\Omega(z)+
\theta_{z}\left\langle \theta_{z},\theta_{z},\theta_{z}\right
\rangle^{\Omega}=0.
\end{equation*}
By chain rule, we then have
\begin{equation*}
\theta_{z}\left(\int_{\tilde{Y}^{\vee}} \Omega(z)\wedge 
\theta_{z}^{3}\Omega(z)\right)
-\int_{\tilde{Y}^{\vee}} \Omega(z)\wedge \theta_{z}^{4}\Omega(z)
+\theta_{z}\left\langle \theta_{z},\theta_{z},
\theta_{z}\right\rangle^{\Omega}=0;
\end{equation*}
in other words,
\begin{equation*}
2\theta_{z}\left\langle \theta_{z},\theta_{z},\theta_{z}\right\rangle^{\Omega}
-\int_{\tilde{Y}^{\vee}} \Omega(z)\wedge \theta_{z}^{4}\Omega(z)=0.
\end{equation*}
Manipulating the Picard--Fuchs equation
\eqref{eq:pf-eq},
we obtain
\begin{equation*}
\label{equation:differential-equation-for-yukawa-couling}
\theta_{z}\left\langle \theta_{z},\theta_{z},\theta_{z}\right\rangle^{\Omega}
=\frac{z}{1-z}\left\langle \theta_{z},\theta_{z},\theta_{z}
\right\rangle^{\Omega}.
\end{equation*}
We can solve the above equation and get
\begin{equation*}
\left\langle \theta_{z},\theta_{z},\theta_{z}
\right\rangle^{\Omega} = \frac{C}{1-z}~\hspace{0.1in}
\mbox{for some constant \(C\)}.
\end{equation*}
One can check 
the \emph{normalized Yukawa coupling}
\begin{equation*}
\left\langle \theta_{z},\theta_{z},\theta_{z}\right\rangle
:=\int_{\tilde{Y}^{\vee}} \frac{\Omega(z)}{\omega_{0}(z)}\wedge \theta_{z}^{3}
\left(\frac{\Omega(z)}{\omega_{0}(z)}\right)
\end{equation*}
is given by
\begin{equation}
\label{equation:normal-yukawa-coupling}
\left\langle \theta_{z},\theta_{z},\theta_{z}\right\rangle = 
\frac{C}{(1-z)\omega_{0}(z)^{2}},
\end{equation}
where \(\omega_{0}(z)\) is the holomorphic series solution 
\eqref{eq:hol-series-sol}.
Consider the deformed series
\begin{equation*}
\label{eq:gamma-hol-series-deformed}
\omega_0(z;\rho):=\sum_{n\ge 0} 
\frac{\Gamma(4n+4\rho+1/2)}{\Gamma(1/2)\Gamma(n+\rho+1)^{4}} z^{n+\rho} 
\end{equation*}
and its derivative with respect to \(\rho\)
\begin{equation*}
\omega_{1}(z):=\left.
\frac{\mathrm{d}}{\mathrm{d}\rho}\right|_{\rho=0} \omega_{0}(z;\rho).
\end{equation*}
Recall that the mirror map is given by
\begin{equation}
\label{equation:mirrir-map-4h}
q = \exp\left(2\pi\sqrt{-1} t\right),~t = 
\frac{1}{2\pi\sqrt{-1}}\frac{\omega_{1}(z)}{\omega_{0}(z)}.
\end{equation}
Using the classical product, one finds \( C = 2 \) in 
\eqref{equation:normal-yukawa-coupling} and
the ``mirror map'' is
\begin{equation*}
q = \frac{z}{256} + \frac{247z^{2}}{1024} + 
\frac{13368541z^{3}}{524288}+\cdots
\end{equation*}
whose inverse is given by 
\begin{equation*}
z = 256 q - 4046848 q^2 + 18282602496 q^3 + \cdots.
\end{equation*}
The (expected) \(A\) model correlation function is
\begin{align}
\begin{split}
\label{eq:a-model-predicted-correlations}
&\langle H,H,H\rangle(q) \\
&= 2 + 29504 q + 1030708800 q^2 + 38440454795264 q^3 + \cdots.
\end{split}
\end{align}

\subsubsection{An instanton calculation}
To complete the quantum test, we will compute the
oribifold Gromov--Witten invariants of \(Y\)
and compare them with \eqref{eq:a-model-predicted-correlations}.

Let \([z_1\mathpunct{:}\ldots\mathpunct{:}z_4]\) be the homogeneous 
coordinates on \(X=\mathbf{P}^{3}\) as before and
\(f\) be a degree \(4\) polynomial in \(\mathbf{P}^{4}\)
such that \(\{f=0\}\cup\bigcup_{i=1}^{4}\{z_{i}=0\}\) 
is the branch locus of the double cover \(Y\to X\).
Consider the graph map
\begin{equation*}
\Gamma_{f}\colon X\to\mathbf{P}(1,1,1,1,4),~
[z_1\mathpunct{:}\ldots\mathpunct{:}z_4]\mapsto
[z_1\mathpunct{:}\ldots\mathpunct{:}z_4\mathpunct{:}f(z)].
\end{equation*}
This is well-defined since \(f\) is of degree \(4\) and
the branch divisor is SNC.
Obviously, \(\Gamma_{f}\) defines an embedding
\(X\hookrightarrow\mathbf{P}(1,1,1,1,4)\).

Let \([y_1\mathpunct{:}\ldots\mathpunct{:}y_5]\)
be the homogeneous coordinate on \(\mathbf{P}(1,1,1,1,4)\).
Consider the covering map
\begin{equation}
\label{eq:Phi}
\Phi\colon\mathbf{P}(1,1,1,1,4)\to\mathbf{P}(1,1,1,1,4),~
[y_1\mathpunct{:}\ldots\mathpunct{:}y_5]
\mapsto [y_1^2\mathpunct{:}\ldots\mathpunct{:}y_5^2].
\end{equation}
Let \(Y'\subset\mathbf{P}(1,1,1,1,4)\)
be the subvariety defined by the degree \(8\) polynomial
\(y_{5}^{2}-f(y_{1}^{2},\ldots,y_{4}^{2})\).
It is clear that \(Y'\) is a smooth 
Calabi--Yau hypersurface. 

Look at the diagram
\begin{equation}
\begin{tikzcd}
  &  &\mathbf{P}(1,1,1,1,4)\ar[d,"\Phi"]\\
  & X \ar[r,"\Gamma_{f}"] &\mathbf{P}(1,1,1,1,4)
\end{tikzcd}
\end{equation}
Taking the fibred product, we obtain a cover
\(Y'\to X\) branched along 
\begin{equation*}
\{f=0\}\cup\bigcup_{i=1}^{4}\{z_{i}=0\}.
\end{equation*}
Let \(\mu_{2}=\{-1,1\}\subset\mathbb{C}^{\ast}\).
We define an action of \(\mu_{2}^{5}\) on \(\mathbf{P}(1,1,1,1,4)\) 
\begin{equation*}
g\cdot [y_1\mathpunct{:}\ldots\mathpunct{:}y_{5}]:=
\left[{g_1}\cdot y_1\mathpunct{:}\ldots\mathpunct{:}g_{5}\cdot y_{5}\right]~
\mbox{where}~g=(g_{1},\ldots,g_{5})\in\mu_{2}^{5}.
\end{equation*}
Notice that the subgroup \(K:=\langle(-1,-1,-1,-1,1)\rangle\)
acts trivially on \(\mathbf{P}(1,1,1,1,4)\). 
Then \(G=\mu_{2}^{5}\slash K\) 
is the Galois group of the
cover $Y'\to X$.
The map 
\begin{equation*}
\textstyle\mu_{2}^{5}\to \mu_{2},~
(g_{1},\ldots,g_{5})\mapsto \prod_{i=1}^{5} g_i
\end{equation*} 
induces a map \(G\to \mu_{2}\). Let \(G'\) be its kernel.
Explicitly,
\begin{equation*}
\textstyle G' = \left\{(g_{1},\ldots,g_{5})\in \mu_{2}^{5}
~\Big|~\prod_{i=1}^{5} g_{i}=1\right\}
\Big\slash K.
\end{equation*}
One can prove the following lemma.
\begin{lemma}
The map
\begin{equation*}
\mathbf{P}(1,1,1,1,4)\slash G'\to \mathbf{P}(1,1,1,1,4)\slash G
\cong \mathbf{P}(1,1,1,1,4)
\end{equation*}
is a double cover
branched along the union of all toric divisors.
\end{lemma}

\begin{corollary}
We have \(Y\simeq Y'\slash G'\). 
\end{corollary}
\begin{proof}
Both $Y$ and $Y'\slash G'$ are double covers
over $\mathbf{P}^3$ having the same 
branch locus. Since the Picard group of $\mathbf{P}^3$
is torsion free, $Y$ and $Y'\slash G'$ must be isomorphic.
\end{proof}

We can regard \(Y'\slash G'\) as
a Calabi--Yau hypersurface in \(\mathbf{P}(1,1,1,1,4)\slash G'\).
The orbifold Gromov--Witten invariants of \(Y\cong Y'\slash G'\) 
can be computed by applying the 
orbifold quantum hyperplane section
theorem 
\cite{2010-Tseng-orbifold-quantum-riemann-roch-
lefschetz-and-serre}*{Theorem 5.2.3}.
We will prove the following theorem.
\begin{theorem}
\label{thm:main-theorem-4h}
The equation \eqref{eq:a-model-predicted-correlations}
is the generating series of the 
untwisted genus zero orbifold Gromov--Witten invariants of \(Y\) 
with all insertions \(H\), where \(H\) is the pullback of 
the hyperplane class
of \(X\). 
\end{theorem}
Now let us prove Theorem \ref{thm:main-theorem-4h}.
We have the following commutative diagram:
\begin{equation}
\label{eq:comm-diag-quotient}
\begin{tikzcd}
\mathbf{P}(1,1,1,1,4)\ar[rd,"q"]\ar[dd,"\Phi"] &\\
& \mathbf{P}(1,1,1,1,4)\slash G'\ar[ld,"p"]\\
\mathbf{P}(1,1,1,1,4) &
\end{tikzcd}
\end{equation}
In the above diagram, \(\Phi\)
is defined in \eqref{eq:Phi}, \(q\)
is the quotient map and \(p\)
is the induced double cover.

\begin{remark}
\label{rmk:toric-data}
\(\mathbf{P}(1,1,1,1,4)\slash G'\)
is a toric variety. We can describe its fan structure as follows.
Consider a rank four lattice \(\overline{N}:=\mathbb{Z}^{4}\)
and integral vectors
\begin{align*}
\rho_{1}&=(1,1,-1,1),\\
\rho_{2}&=(-1,1,-1,1),\\
\rho_{3}&=(1,-1,1,1),\\
\rho_{4}&=(-1,1,1,-1),\\
\rho_{5}&=(3,-5,-3,1).
\end{align*}
Note that \(\rho_{1}+\rho_{2}+\rho_{3}+4\rho_{4}+\rho_{5}=0\) and
that \(\rho_{1},\ldots,\rho_{5}\)
generate a sublattice \(\overline{N}'\) of index \(8\) in 
\(\mathbb{Z}^{4}\). Put
\begin{equation*}
\sigma_{j}:=\mathrm{Cone}\{\rho_{1},\cdots,
\hat{\rho}_{j},\cdots,\rho_{5}\},~
j=1,\ldots,5.
\end{equation*}
All the \(\sigma_{j}\) together with all their faces 
form a complete fan \(\Xi\) in \(\overline{N}'\otimes\mathbb{R}\).
With respect to \(\overline{N}'\), the toric variety \(X_{\Xi,\overline{N}'}\)
is isomorphic to \(\mathbf{P}(1,1,1,1,4)\).
Moreover, we can show that
\(G'\cong \overline{N}\slash \overline{N}'\) 
and \(X_{\Xi,\overline{N}}\cong \mathbf{P}(1,1,1,1,4)\slash G'\).

\end{remark}

Let \(\mathfrak{X}=[\mathbf{P}(1,1,1,1,4)\slash G']\)
be the quotient stack.
The coarse moduli space is denoted by \(|\mathfrak{X}|\)
(\(\cong\mathbf{P}(1,1,1,1,4)\slash G'\)).
Let us write down the (non-extended) \(I\)-function for \(\mathfrak{X}\).
Denote by \(D_{i}\) the toric divisor associated 
to the \(1\)-cone \(\mathbb{R}_{\ge}\rho_{i}\)
defined in Remark \ref{rmk:toric-data}.
One can easily prove that \(D_{1}\equiv D_{2}\equiv D_{3}\equiv D_{5}\) and 
\(D_{4}\equiv 4D_{5}\). 
and \(D_{1}.\ell = 1/4\). Write \(H=D_{1}\).
In the present case, \(\mathrm{H}^{2}(|\mathfrak{X}|;\mathbb{C})=
\mathbb{C}\cdot H\)
and \(\mathrm{H}_{2}(|\mathfrak{X}|,\mathbb{Z})=\mathbb{Z}\langle\ell\rangle\)
where \(\ell\) is the curve class coming from a \emph{wall} in \(\Xi\).
Then \(8H\) is a Cartier divisor on \(|\mathfrak{X}|\).
The non-extended \(I\)-function (on
the very small parameter space \(\mathrm{H}^{\le 2}(|\mathfrak{X}|;\mathbb{C})
\subset \mathrm{H}_{\mathrm{CR}}^{\bullet}(\mathfrak{X};\mathbb{C})\)) 
is given by 
\begin{align*}
I_{\mathfrak{X}}(t;z)&=
z\cdot \exp({\textstyle H t\slash z})\sum_{g\in \mathrm{C}(G')}
\sum_{d\in\overline{\mathrm{NE}}_{g}}q^{d}\prod_{j=1}^{5}
\frac{\prod_{\langle d \rangle=\langle m\rangle,~m\le 0} (D_{j}+mz)}{\prod_{\langle d \rangle=\langle m\rangle,~m\le d}(D_{j}+mz)}\cdot\mathbf{1}_{g}\\
&=z\cdot \exp(Ht\slash z)\sum_{g\in \mathrm{C}(G')}
\sum_{d\in\overline{\mathrm{NE}}_{g}}q^{d}
\frac{1}{\displaystyle\prod_{\substack{\langle d \rangle=\langle m\rangle\\0< m\le d}}(H+mz)^{4}
\prod_{\substack{\langle d \rangle=\langle m\rangle\\0< m\le 4d}}(4H+mz)}\cdot\mathbf{1}_{g}
\end{align*}
where \(\mathrm{C}(G')\) denotes the set of conjugacy classes of \(G'\)
and \(\mathbf{1}_{g}\) is the unit in the cohomology ring of the component 
associated to \(g\) 
(cf.~\cite{2015-Coates-Corti-Iritani-Tseng-a-mirror-theorem-for-toric-stacks}).
In our case, \(G'\) is a finite abelian 
group and \(\mathrm{C}(G')= G'\).

Now we can apply the orbifold quantum Lefschetz 
hyperplane theorem to compute 
the orbifold Gromov--Witten invariants for \(Y\).
Recall that 
\(Y\cong Y'\slash G'\) is an anti-canonical hypersurface in \(\mathfrak{X}\).
Applying the \emph{hypergeometric modification} trick to 
\(I_{\mathfrak{X}}\) and restricting the result to the 
\emph{untwisted sector} \(\mathbf{1}_{e}\), we obtain
\begin{align}
\label{eq:i-function-modification-untwisted-4h}
I^{\mathrm{untw}}_{\mathfrak{X}}(t;z)
&=z\cdot \exp(Ht\slash z)
\sum_{d\in\mathbb{Z}_{\ge 0}}q^{d}
\frac{\displaystyle\prod_{1\le m\le 8d}(8H+mz)}{\displaystyle\prod_{1\le m\le d}(H+mz)^{4}
\prod_{1\le m\le 4d}(4H+mz)}\cdot \mathbf{1}_{e}
\end{align}
The series \eqref{eq:i-function-modification-untwisted-4h}
is almost identical to 
\begin{equation*}
\tilde{I}_{Y'}(t;z)
=z\cdot \exp(ht\slash z)
\sum_{d\in\mathbb{Z}_{\ge 0}}q^{d}
\frac{\displaystyle\prod_{1\le m\le 8d}(8h+mz)}{\displaystyle\prod_{1\le m\le d}(h+mz)^{4}
\prod_{1\le m\le 4d}(4h+mz)}
\end{equation*} 
the hypergeometric modification of the \(I\)-function
for the Calabi--Yau hypersurface \(Y'\) in \(\mathbf{P}(1,1,1,1,4)\). 
Here \(h\) is the hyperplane class of \(\mathrm{P}(1,1,1,1,4)\). The
only difference is the hyperplane classes \(h\) and \(H\).

\begin{corollary}
The mirror maps for \(8h\cdot\tilde{I}_{Y'}(t;z)\)
and \(8H\cdot\tilde{I}^{\mathrm{untw}}_{\mathfrak{X}}(t;z)\) are identical
if we treat \(H\) and \(h\) as formal variables such that \(h^{5}=H^{5}=0\).
\end{corollary}

Now we investigate the Poincar\'{e} pairing on \(Y'\)
and the orbifold Poincar\'{e} pairing
\(Y\). Let \(h\) and \(H\) be the hyperplane 
classes on \(\mathbf{P}(1,1,1,1,4)\)
and \(|\mathfrak{X}|\) as before.
By abuse of notation,  
the restriction of \(h\) and \(H\)
to \(Y'\) and \(Y\) are also denoted by \(h\) and \(H\).
From \eqref{eq:comm-diag-quotient}, we see that \(H=p^{\ast}h\) and
\(q^{\ast}H=2h\) (note that \(\Phi^{\ast}h=2h\)). Therefore,
\begin{equation*}
\int_{Y} H^{3}=\frac{1}{8}\int_{Y'} (2h)^{3}=
\int_{Y'}h^{3}=\int_{\mathbf{P}(1,1,1,1,4)} 8h^{4}=2.
\end{equation*}
We see that \(\{\mathbf{1},H,H^{2}/2,H^{3}/2\}\) is a symplectic basis
of \(\mathrm{H}^{even}(Y;\mathbb{C})\) with
respect to the orbifold Poincar\'{e} pairing. 
On the other hand, we know that 
\(\{\mathbf{1},h,h^{2}/2,h^{3}/2\}\) is a symplectic basis 
of \(\mathrm{H}^{even}(Y';\mathbb{C})\) with respect to the 
Poincar\'{e} pairing on \(Y'\).
This shows that
\begin{equation}
\mathrm{H}^{even}(Y';\mathbb{C}) \to
\mathrm{H}^{even}(Y;\mathbb{C}),~h^{k}\mapsto H^{k}
\end{equation}
is an isomorphism between normed linear spaces
(with respect to Poincar\'{e} paring
and orbifold Poincar\'{e} pairing).
Since the \(J\)-function of \(Y'\)
is identical to the restriction 
of the untwisted part of the \(J\)-function
of \(Y\) to the 
very small parameter space, 
we conclude the proof of Theorem \ref{thm:main-theorem-4h}.

\begin{remark}
We can also consider the double cover
\(Y\to X:=\mathbf{P}^{3}\) branched along eight
hyperplanes in general position. In this case, we have \(r=4\)
and the nef-partition is \(-K_{X}=H+H+H+H\).
Denote by \(Y^{\vee}\) the singular mirror.
As before, we can compute \(h^{2,1}(Y^{\vee})=1\).
We can carry out the Picard--Fuchs equation,
the unique holomorphic (affine) period as well
as the mirror map. The Picard--Fuchs equation turns 
out to be
\begin{equation*}
\theta_{z}^{4} - z \left(\theta_{z}+\frac{1}{2}\right)^{4}=0.
\end{equation*}
We obtain the predicted \(A\) model correlation function
\begin{align*}
\begin{split}
\label{eq:a-model-predicted-correlations}
&\langle H,H,H\rangle(q) \\
&= 2 + 64 q + 9792 q^2 + 1404928 q^3 + 205641280 
q^4 + 30593496064 q^5 + \cdots.
\end{split}
\end{align*}
The series was also obtained by E.~Sharpe in
\cite{2013-Sharpe-predictions-for-gromov-witten-
invariants-of-noncommutative-resolutions}
using the technique of gauged linear sigma models (GLSMs).
We can prove that it is the generating series of the 
untwisted genus zero orbifold Gromov--Witten invariants of \(Y\) 
with all insertions \(H\), where \(H\) is the pullback of 
the hyperplane class
of \(X=\mathbf{P}^{3}\). 
\end{remark}

\subsubsection{Geometric transitions}
A \emph{geometric transition} 
is a complex degeneration \(W\rightsquigarrow W_{0}\)
followed by a resolution \(W_{0}\leftarrow Z\).
In \cite{1999-Morrison-through-the-looking-glass},
Morrison conjectured that geometric transitions
are reversed under mirror symmetry.
To be precise, let \(W\) and \(Z\) be Calabi--Yau manifolds
and \(W^{\vee}\) and \(Z^{\vee}\) be their mirror. Suppose
that \(W\rightsquigarrow W_{0} \leftarrow Z\)
is a geometric transition. Then 
the conjecture asserts that there exists
a geometric transition \(Z^{\vee}\rightsquigarrow W_{0}^{\vee}
\leftarrow W^{\vee}\) connecting \(W^{\vee}\) and \(Z^{\vee}\).
In this paragraph,
we shall see that our singular mirror construction
fits into the picture nicely.

Our singular double cover \(Y\) admits 
a smoothing to \(W\) by deforming the
branch locus into a smooth 
degree \(8\) hypersurface in \(\mathbf{P}^{3}\)
and \(W\) can be realized as a smooth Calabi--Yau
hypersurface in \(\mathbf{P}(1,1,1,1,4)\). 
If we put
\begin{align*}
\Delta_{1} = \mathrm{Conv}\{(6,-2,-2,-1),&(-2,6,-2,-1),(-2,-2,6,-1),\\
&(0,0,0,1), (-2,-2,-2,-1)\},
\end{align*}
then \(\mathbf{P}(1,1,1,1,4)=\mathbf{P}_{\Delta_{1}}\).

The matrix 
\begin{equation*}
\begin{bmatrix}
1 & 1 & 0 & 0\\
-1& 0 & 0 & 1\\
0 & 0 & 1 & 1\\
0 & 0 & 0 &-1
\end{bmatrix}\in\mathrm{GL}_{4}(\mathbb{Z})
\end{equation*}
takes \(\{\rho_{1},\rho_{2},\rho_{3},\rho_{4},\rho_{5}\}\) into
\begin{align*}
\nu_{1}&=(2,0,0,-1),\\
\nu_{2}&=(0,2,0,-1),\\
\nu_{3}&=(0,0,2,-1),\\
\nu_{4}&=(0,0,0,1),\\
\nu_{5}&=(-2,-2,-2,-1).
\end{align*}
Let \(\nabla_{2}:=\mathrm{Conv}\{\nu_{1},\ldots,\nu_{5}\}\).
It is easy to check that \(\nabla_{2}\) is reflexive
and 
\begin{align*}
\Delta_{2}:=\nabla_{2}^{\vee}=\mathrm{Conv}
\{(3,-1,-1,-1),&(-1,3,-1,-1),(-1,-1,3,-1),\\
&(0,0,0,1),(-1,-1,-1,-1)\}.
\end{align*}
In other words, \(Y\) is a Calabi--Yau hypersurface
in \(\mathbf{P}_{\Delta_{2}}\). Denote by
\(\widehat{\mathbf{P}}_{\Delta_{2}}\to \mathbf{P}_{\Delta_{2}}\)
an MPCP desingularization and \(Z\to Y\) be the 
induced partial resolution. We obtain the following diagram.
\begin{equation}
\label{diag:transition}
\begin{tikzcd}
&   &Z\subset \widehat{\mathbf{P}}_{\Delta_{2}}\ar[d]\\
& \mathbf{P}_{\Delta_{1}}\supset W\ar[r,rightsquigarrow] & 
Y\subset \mathbf{P}_{\Delta_{2}}.
\end{tikzcd}
\end{equation}
One can obtain a MPCP desingularization 
\(\widehat{\mathbf{P}}_{\Delta_{1}}\to \mathbf{P}_{\Delta_{1}}\)
by taking weighted blow-up at the singular point.
\(W\) can be also regarded as a Calabi--Yau hypersurface in
\(\widehat{\mathbf{P}}_{\Delta_{1}}\).

On the other hand, we can construct the mirror of \(W\) and \(Z\)
by taking their dual polytope. Note that
\begin{align*}
\nabla_{1}:=\Delta_{1}^{\vee}=
\mathrm{Conv}\{(1,0,&0,-1),(0,1,0,-1),(0,0,1,-1),\\
&(0,0,0,1),(-1,-1,-1,-1)\}.
\end{align*}
By our construction, \(Y^{\vee}\)
is a double cover over 
an MPCP desingularization \(X^{\vee}\to \mathbf{P}_{\nabla}\). 
It turns out that \(Y^{\vee}\) can be realized 
as a Calabi--Yau hypersurface in some toric variety 
which we now describe. 

\(Y^{\vee}\to X^{\vee}\) is a double cover
branched along the union of toric divisors and a
general section \(s\in \mathrm{H}^{0}(X^{\vee},-K_{X^{\vee}})\).
We have an embedding \(\Gamma_{s}\colon X^{\vee}
\hookrightarrow\mathbf{P}_{X^{\vee}}
(\mathbb{L}\oplus\mathbb{C})\) via 
the graph of \(s\).
Here \(\mathbb{L}\)
is the total space of the anticanonical bundle
of \(X^{\vee}\). 
The fan of the toric variety 
\(\mathbf{P}_{X^{\vee}}
(\mathbb{L}\oplus\mathbb{C})\)
is easy to describe. Let
\(\tau\) be a maximal cone in the fan defining \(X^{\vee}\)
generated by \(\eta_{1},\eta_{2},\eta_{3}\in \Delta\cap M\).
We put
\begin{align*}
\tau_{0} &= 
\mathrm{Cone}\{(\eta_{1},-1),(\eta_{2},-1),(\eta_{3},-1),(0,0,0,1)\},\\
\tau_{\infty} &= 
\mathrm{Cone}\{(\eta_{1},-1),(\eta_{2},-1),(\eta_{3},-1),(0,0,0,-1)\}.
\end{align*}
The collection of \(\tau_{0}\) and \(\tau_{\infty}\)
together with all their faces with \(\tau\)
running through all maximal cones in the fan for \(X^{\vee}\)
is the fan defining \(\mathbf{P}_{X^{\vee}}
(\mathbb{L}\oplus\mathbb{C})\).
We denote it by \(\Theta\).
Then \(\Theta\) is a fan in \(\overline{M}_{\mathbb{R}}\), where
\(\overline{M}:=M\times\mathbb{Z}\).
Consider three sublattices
\begin{equation*}
\overline{M}_{1}:=2M\times 2\mathbb{Z} \subset 
\overline{M}_{2}:=M\times 2\mathbb{Z}\subset 
\overline{M}:=M\times\mathbb{Z}.
\end{equation*}

\begin{proposition}
\(Y^{\vee}\) can be realized as a Calabi--Yau
hypersurface in \(X_{\Theta,\overline{M}_{2}}\).
\end{proposition}
\begin{proof}
Note that 
\begin{equation*}
X_{\Theta,\overline{M}} \cong X_{\Theta,\overline{M}_{1}} \cong
\mathbf{P}_{X^{\vee}}(\mathbb{L}\oplus\mathbb{C})
\end{equation*}
and the inclusion \(\overline{M}_{1}\subset\overline{M}\)
induces a finite cover \(X_{\Theta,\overline{M}_{1}}
\to X_{\Theta,\overline{M}}\).
We obtain the following diagram:
\begin{equation*}
\begin{tikzcd}
X_{\Theta,\overline{M}_{1}}\ar[rd,"q"]\ar[dd,"\Phi"] &\\
& X_{\Theta,\overline{M}_{2}}\ar[ld,"p"]\\
X_{\Theta,\overline{M}} &
\end{tikzcd}
\end{equation*}
Put \(G:=\overline{M}_{2}\slash\overline{M}_{1}\).
We have 
\begin{equation*}
X_{\Theta,\overline{M}_{1}}\slash G\cong X_{\Theta,\overline{M}_{2}}.
\end{equation*}
Moreover, \(p\) is a double cover branched
along the union of toric divisors in \(X_{\Theta,\overline{M}}\).
Let \(S\) be the fibred product
\begin{equation*}
\begin{tikzcd}
& S\ar[r]\ar[d] & X_{\Theta,\overline{M}_{1}}\ar[d,"\Phi"]\\
& X^{\vee} \ar[r,"\Gamma_{s}"] &X_{\Theta,\overline{M}_{1}}
\end{tikzcd}
\end{equation*}
It then follows that \(Y^{\vee}\cong S\slash G
\subset X_{\Theta,\overline{M}_{2}}\).
\end{proof}
One can check that 
\(\Theta\) (with respect to the integral structure \(\overline{M}_{2}\))
is a refinement of the normal fan of \(\nabla_{1}\).
In other words, there exists a MPCP desingularization
\(\widehat{\mathbf{P}}_{\nabla_{1}}\) which
dominants \(X_{\Theta,\overline{M}_{2}}\).

On the other hand, one can again deform 
the branch locus of \(Y^{\vee}\to X^{\vee}\)
to obtain a smooth Calabi--Yau
hypersurface \(Z^{\vee}\) in the toric variety
\(X_{\Theta,\overline{M}}\)
which can be viewed as an MPCP 
desingularization of \(\mathbf{P}_{\nabla_{2}}\).
We thus obtain the following diagram
mirror to \eqref{diag:transition}
\begin{equation*}
\begin{tikzcd}
& &Z^{\vee}\subset \widehat{\mathbf{P}}_{\nabla_{2}}\ar[d,rightsquigarrow]\\
& \widehat{\mathbf{P}}_{\nabla_{1}}\supset W^{\vee}\ar[r] 
& Y^{\vee}\subset X_{\Theta,\overline{M}_{2}}.
\end{tikzcd}
\end{equation*}


\end{document}